%% file: main.tex
\documentclass[letterpaper, 10 pt, conference]{ieeeconf}  

\IEEEoverridecommandlockouts                              
\title{\LARGE \bf
A Bandit Learning Method for Continuous Games under Feedback Delays with Residual Pseudo-Gradient Estimate 
}

\usepackage{hyperref}
\usepackage{url}
\usepackage{amssymb}

\usepackage[dvipsnames]{xcolor}
\usepackage{mathrsfs}
\usepackage{bbm, dsfont}
\let\labelindent\relax 
\usepackage{amsmath, amsfonts, outlines, mathtools, outlines, cancel, enumitem, float, stackengine, graphicx, romannum}
\usepackage{subcaption}
\usepackage[english]{babel}
\usepackage[utf8]{inputenc}
\usepackage[T1]{fontenc}
\usepackage{newtxtext,newtxmath}
\usepackage{algorithm, algpseudocode}

\setenumerate[1]{label=(\roman*)}

\newtheorem{theorem}{Theorem}

\newtheorem{assumption}{Assumption}
\newtheorem{lemma}{Lemma}

\newtheorem{envdef}{Definition}

\DeclareMathSizes{10}{9}{6}{5}
\allowdisplaybreaks

\input{macros.tex}
\author{Yuanhanqing Huang$^{1}$ and Jianghai Hu$^{1}$
\thanks{This work was supported by the National Science Foundation under Grant No. 2014816 and No.2038410. }
\thanks{$^{1}$The authors are with the Elmore Family School of Electrical and Computer Engineering, Purdue University, West Lafayette, IN, 47907, USA 
        {\tt\small \{huan1282, jianghai\}@purdue.edu}}%
}

\newif\ifproceeding
\newif\ifarxiv

\proceedingfalse
\arxivtrue

\begin{document}

\maketitle
\thispagestyle{empty}
\pagestyle{empty}

\begin{abstract}
Learning in multi-player games can model a large variety of practical scenarios, where each player seeks to optimize its own local objective function, which at the same time relies on the actions taken by others. 
Motivated by the frequent absence of first-order information such as partial gradients in solving local optimization problems and the prevalence of asynchronicity and feedback delays in multi-agent systems, we introduce a bandit learning algorithm, which integrates mirror descent, residual pseudo-gradient estimates, and the priority-based feedback utilization strategy, to contend with these challenges. 
We establish that for pseudo-monotone plus games, the actual sequences of play generated by the proposed algorithm converge a.s. to critical points. 
Compared with the existing method, the proposed algorithm yields more consistent estimates with less variation and allows for more aggressive choices of parameters. 
Finally, we illustrate the validity of the proposed algorithm through a thermal load management problem of building complexes.

\end{abstract}

\section{INTRODUCTION}

With the proliferation of cyber-physical engineering systems and modern network applications, the non-cooperative multi-player game has emerged as a valuable tool for modeling and investigating the decision-making process of agents with interest conflicts \cite{li2022confluence}. 
Each participant in the game seeks to unilaterally optimize its own objective, whose value also depends on the action taken by others. 
Notable practical applications include thermal load management of autonomous buildings \cite{jiang2021game}, supply-side risk management in power markets \cite{kannan2013addressing}, power control in wireless communication \cite{zhou2021robust}, path planning and control of self-driving cars \cite{liniger2019noncooperative}, etc. 

Over the past few decades, the control and optimization communities have devoted significant effort to developing solution algorithms for non-cooperative games by reformulating them as variational inequalities \cite{facchinei2003finite}. 
Recently, there has been growing interest in distributed solutions under partial information settings, as they offer advantages in scalability and privacy preservation \cite{pavel2019distributed, bianchi2022fast, huang2022distributed}. 
Despite their promise in some cases, the applicability of these methods is often limited by the requirement for the existence of first-order/pseudo-gradient oracles or the full knowledge of the objectives, which may not be available in practical settings. 
Prompted by the need to relax the information requirement, researchers approximate the missing pseudo-gradient information with the actions taken and the resulting objective values. 
This problem can then be fit into the framework of bandit online learning \cite{shalev2012online}, where at every updating step, each player selects an action, observes the realized objective value, and updates its strategy according to the observed result and the process repeats. 

Another practical challenge that hinders the implementation in real-world scenarios is the latency between taking action and receiving bandit feedback, which is further exacerbated in multi-agent systems, where agents could experience heterogeneous delays. 
Latency can arise as a result of significant communication delays or the fundamental limitation that certain actions take time to manifest their effects. 
In the context of routing problems \cite{vu2021fast}, assessing the effectiveness of a navigation strategy entails waiting for a driver to execute the instructions, operate the vehicle, and record the time elapsed. 
In light of the preceding consideration, the primary objective of this work is to propose a bandit online learning algorithm for multi-player continuous games that can ensure convergence despite the presence of feedback delays.  

\textit{Related Work:} 
\Tblue{
In the context of bandit learning in games with instantaneous feedback, Bravo et al. \cite{bravo2018bandit} introduced a bandit mirror descent (MD) method that ensures a.s. convergence when the game is strictly monotone. 
The single-point pseudo-gradient estimate is obtained via the simultaneous perturbation stochastic approximation (SPSA) approach \cite{agarwal2010optimal}. 
In the context of strongly monotone games and their variants, the algorithms proposed in \cite{lin2021optimal, tatarenko2022rate, tatarenko2023convergence, drusvyatskiy2022improved} similarly employ single-point estimates of the pseudo-gradient and attain a $\mathcal{O}(1/t^{1/2})$ convergence rate.  
The single-point estimates are also applied in \cite{tatarenko2020bandit} and \cite{gao2022bandit} for merely monotone games and their variants. 
Given the susceptibility of single-point estimates to large variances, a critical factor impacting the efficiency of algorithms, Tatarenko et al. \cite{tatarenko2022rate} introduced the two-point estimate. 
This strategy mitigates variance-related issues and enhances the convergence rate to $\mathcal{O}(1/t)$ for strongly monotone games. 
In the field of zeroth-order optimization, Zhang et al. \cite{zhang2022new} considered a residual feedback scheme to control the estimation variance. 
By integrating residual pseudo-gradient estimate into the single-call extra-gradient scheme, Huang et al. \cite{huang2023zeroth} developed two bandit algorithms.
The proposed algorithms only require a single query per iteration and ensure a.s. convergence for pseudo-monotone plus games and achieve $O(1/t^{1 - \epsilon})$ convergence rate for strongly monotone games. 
}

To contend with the feedback delays in games, Huang et al. \cite{huang2022convergence} proposed an algorithm based on the improved accelerated gradient descent for potential games, which can tackle cases ranging from sublinear delays to superlinear delays. 
Zhang et al. \cite{zhang2022multi} focused on the general-sum Markov games where the agents are impacted by heterogeneous reward delays and proposed the delay-adaptive multi-agent V-learning to procure coarse-correlated equilibria. 
Of particular relevance is \cite{heliou20agradient}, in which Helious et al. delved into the development of a no-regret bandit learning algorithm for strictly monotone games corrupted by homogeneous sublinear reward delays. 
Nevertheless, the delicate balance between bias and variance of the proposed method is elusive and requires careful calibration. 
Moreover, its stringent requirements on step sizes and query radius hinder its applicability. 

\textit{Contributions:} 
\Tblue{
First, we propose a bandit learning algorithm under feedback delays, where the delays can be heterogeneous but upper-bounded by a constant or homogeneous with a sublinearly growing upper bound. 
Our algorithm integrates mirror descent, residual pseudo-gradient estimates, and the priority-based feedback utilization strategy. 
It is the first algorithm that employs the variance control strategy via single-point residual estimates in the scenario of bandit learning with delays. 
Second, we establish the a.s. convergence of the proposed algorithm for pseudo-monotone plus games. 
While some of the proving techniques have been previously established in \cite{huang2023zeroth}, this paper places additional emphasis on addressing the error caused by delays, which can complicate the problem, particularly when two subsequent realized objective values are required for each single estimate. 
Compared to the existing method in \cite{heliou20agradient}, the proposed algorithm in this work maintains a constant upper bound for the estimation variance and relaxes the conditions on step size and query radius by incorporating the residual pseudo-gradient estimates. 
In addition, we evaluate the performance of the solution algorithms using the thermal load management problem of buildings. 
Compared to the existing work, the proposed algorithm achieves faster and more consistent convergence. 
\ifproceeding
Due to the page limit, complete proofs are presented in \cite{huang2023bandit}. 
\fi
}

\textit{Basic Notations:} 
For a set of vectors $\{v_i\}_{i \in S}$, $[v_i]_{i \in S}$ or $[v_1; \cdots; v_{|S|}]$ denotes their vertical stack. 
For a vector $v$ and a positive integer $i$, $[v]_i$ denotes the $i$-th entry of $v$. 
Denote $\nset{}{+} \coloneqq \nset{}{} \backslash \{0\}$ and $\rset{}{++} \coloneqq (0, +\infty)$. 
We let $\norm{\cdot}_2$ represent the Euclidean norm, $\norm{\cdot}$ a general norm, and $\norm{\cdot}_*$ its dual.
For a set $\mathcal{S}$, let $\mathds{1}_{\mathcal{S}}$ denote the indicator function for this set, i.e., $\mathds{1}_{\mathcal{S}}(x) = 1$ if $x \in \mathcal{S}$ and $0$ otherwise. 
Let $\cl{\mathcal{S}}$ denote the closure of set $\mathcal{S}$, $\text{int}(\mathcal{S})$ the interior, and $\partial \mathcal{S}$ the boundary. 
The symbols $a \wedge b$ and $a \vee b$ stand for the lesser and the greater of the two real numbers $a$ and $b$, respectively. 

\section{SETUP AND PRELIMINARIES}

\subsection{Problem Setup}
In this subsection, we formalize the multi-player continuous game with feedback delays that we will investigate and introduce the assumptions to impose. 
In this $N$-player game $\mathcal{G}$, with the player set given by $\mathcal{N} \coloneqq \{1, \ldots, N\}$, each player $i$ needs to optimize its own local objective by determining its local action $x^i \in \mathcal{X}^i$, where $\mathcal{X}^i \subseteq \rset{n^i}{}$ represents the local strategy space of player $i$. 
For brevity, we let the stack vector $x \coloneqq [x^j]_{j \in \mathcal{N}}$ denote the global action, the stack vector $x \coloneqq [x^j]_{j \in \mathcal{N}_{-i}}$ denote the action taken by all players except player $i$ with $\mathcal{N}_{-i} \coloneqq \mathcal{N} \backslash \{i\}$. 
Similarly, denote the global strategy space $\mathcal{X} \coloneqq \prod_{j \in \mathcal{N}} \mathcal{X}^j \subseteq \rset{n}{}$ with $n \coloneqq \sum_{j \in \mathcal{N}} n^j$. 
Formally, given the action $x^{-i}$ taken by other players, each player $i$ aims to solve the following local problem: 
\begin{align}
\minimize_{x^i \in \mathcal{X}^i} J^i(x^i; x^{-i}). 
\end{align}
The following conditions are imposed regarding the smoothness of objective $J^i$'s and the properties of $\mathcal{X}^i$'s.

\begin{assumption}\label{asp:objt-set}
For each player $i$, the local objective function $J^i$ is continuously differentiable $(C^1)$ in $x$ over the strategy space $\mathcal{X}$.
The individual strategy space $\mathcal{X}^i$ is compact and convex.
Moreover, each $\mathcal{X}^i$ possesses a non-empty interior. 
\end{assumption}

The underlying probability space is given by $(\Omega, \mathcal{F}, \mathbb{P})$. 
One operator we will leverage throughout is the pseudo-gradient operator $F:\mathcal{X} \to \rset{n}{}$, which is defined as the stack of the partial gradient given the smoothness imposed in Assumption~\ref{asp:objt-set}, i.e., 
\begin{align}
F: x \mapsto [\nabla_{x^i} J^i(x^i; x^{-i})]_{i \in \mathcal{N}}. 
\end{align}
The Lipschitz continuity of $F$ then entails the fact that each $J^i$ is $C^1$ and $\mathcal{X}^i$ compact, i.e., there exists some constant $L$, such that for arbitrary $x$ and $y \in \mathcal{X}$, $\norm{F(x) - F(y)}_* \leq L\norm{x - y}$. 
In the same vein, the gradient $\nabla_x J^i: \mathcal{X} \to \rset{n^i}{}$ is also Lipschitz continuous and admits a tighter Lipschitz constant denoted by $L^i$. 
Throughout this work, we will concentrate on the solution concept known as critical points (CPs) \cite[Section~2]{mertikopoulos2022learning}, whose definition is given as follows. 
\begin{envdef}\label{def:cps} (Critical Points)
A decision profile $x_* \in \mathcal{X}$ is a critical point of the non-cooperative game $\mathcal{G}$ if it solves the associated (Stampacchia) variational inequality (VI), i.e., 
\begin{align}\label{eq:cps}
\langle F(x_*), x - x_*\rangle \geq 0, \; \forall x \in \mathcal{X}, 
\end{align}
which is typically denoted by the abbreviation $\text{VI}(\mathcal{X}, F)$. 
\end{envdef}
Besides, the following assumption is postulated regarding the monotonicity of $F$ to facilitate the convergence analysis. 
\begin{assumption}
The pseudo-gradient $F$ is pseudo-monotone plus on $\mathcal{X}$, i.e., $F$ is pseudo-monotone, i.e., for all $x, y \in \mathcal{X}$, $\langle F(y), x - y\rangle\geq 0\implies \langle F(x), x - y\rangle\geq 0$, and satisfies for any action profiles $x, y \in \mathcal{X}$, $\langle F(y), x-y\rangle \geq 0$ and $\langle F(x), x - y\rangle = 0 \implies F(x) = F(y)$. 
\end{assumption}

\Tblue{Pseudo-monotone plus games are a broader class of games than strictly monotone games, but they are not a subset or a superset of merely monotone games. 
Examples of pseudo-monotone plus games that are not merely monotone can be found in \cite[Section~V.A]{huang2023zeroth}\cite{kannan2019optimal}. }

\subsection{Setup for Feedback Delays}
In this work, we consider the scenario where there exists some time lag between the time when an action is taken and the time when the associated realized objective value is received by the player. 
To simplify notation, we let the realized objective value of player $i$ at the $k$-th iteration be denoted by $\hat{J}^i_k$. 
Then, for player $i$, the delay time of $\hat{J}^i_k$ is denoted by $d^i_k$, and this piece of bandit information is available at iteration $\ceil{k + d^i_k}$.
We impose that the delay time should grow at most sublinearly in the iteration $k$ when the delays are homogeneous or be upper bounded by some constant when the delays are heterogeneous, which is formally stated in the assumptions below. 
\begin{assumption}\label{asp:delay}
For each player $i$, the feedback delay $d^i_k$ associated with the realized objective value $\hat{J}^i_k$ is a random variable and $d^i_k \in [0, \bar{d}(k)]$, where $\bar{d}(k) \coloneqq k^{\alpha_d} + \bar{d}$, for some constants $\bar{d} \geq 0$ and $0 \leq \alpha_d < 1$. 
\end{assumption}

\begin{assumption}\label{asp:delay-plus}
Either one of the following statements holds:
\begin{outline}[enumerate]
\1 the delay $d^i_k$ is upper-bounded by a constant $\bar{d}$;
\1 all the players experience the same delay, i.e., $d^1_k = \cdots = d^N_k = d_k$. 
\end{outline}
\end{assumption}

\Tblue{
The issue of handling delays that grow sublinearly or even superlinearly relative to a global clock is receiving increasing attention in the realm of distributed systems \cite{zhou2022distributed}.
For example, in volunteer computing grids, the participation of new and faster workers in the network can undermine the performance of slower workers, causing their computation requests to accumulate quickly over time and resulting in growing delays. 
}

\subsection{Mirror Map and Mirror Descent}
To streamline our subsequent discussion, we briefly introduce mirror descent and related concepts in this subsection. 
The interested readers are referred to \cite[Ch.~4]{bubeck2014theory} for more detailed information. 
Let $\mathcal{B}$ denote a Banach space and $\mathcal{B}^*$ its dual. 
We first let $\psi:\dom\psi \to \rset{}{}$ with $\dom\psi \subseteq \mathcal{B}$ denote a distance generating function (DGF). 
Here, $\dom\psi$ refers to the set where $\psi$ is well-defined and is assumed to be convex and open. 
The DGF $\psi$ satisfies: 
$(\romannum{1})$ $\psi$ is differentiable and $\Tilde{\mu}$-strongly convex for some $\Tilde{\mu} > 0$; 
$(\romannum{2})$ $\nabla \psi(\dom \psi) = \rset{n}{}$; 
$(\romannum{3})$ $\cl{\dom \psi} \supseteq \mathcal{X}$ and $\lim_{x \to \partial (\dom \psi)}$ $\norm{\nabla \psi(x)}_* = +\infty $. 
With the DGF $\psi$ in hand, the mirror map $\nabla \psi^*$ can be defined as:
\begin{align}
\nabla \psi^*(z) = \argmax_{x \in \mathcal{X}} \{\langle z, x\rangle - \psi(x)\}, 
\end{align}
which can be regarded as an extension of projection in general spaces. 
We let $D(\cdot, \cdot): \mathcal{B} \times \mathcal{B} \to \rset{}{}$ represent the Bregman divergence, whose formal expression is given by: 
\begin{align}
D(p,x) = \psi(p) - \psi(x) - \langle \nabla \psi(x), p - x\rangle, \forall p, x \in \dom \psi.
\end{align}
\begin{assumption}\label{asp:recip}
(Bregman Reciprocity) The chosen DGF $\psi$ satisfies that when the sequence $(x_k)_{k \in \nset{+}{}}$ converges to some point $p$, i.e., $\norm{x_k - p} \to 0$, then $D(p, x_k) \to 0$.
\end{assumption}
The above assumption is introduced to enable the Bregman divergence $D(p, \cdot)$ to function as a specific distance metric with respect to $p$, thereby delineating a particular vicinity around $p$.
The prox-mapping $P_{x, \mathcal{X}}: \mathcal{B}^* \to \dom \psi \cap \mathcal{X}$, induced by the Bregman divergence, is defined as:
\begin{align}
P_{x, \mathcal{X}}(y) = \argmin_{x^\prime \in \mathcal{X}}\{\langle y, x - x^\prime\rangle + D(x^\prime, x)\}, 
\end{align}
which plays an essential role in mirror descent and its variants. 
A lemma characterizing mirror maps and prox-mappings that will be frequently used in the subsequent analysis is given below. 
\begin{lemma}\label{le:md-lips}
Consider the ambient Banach space $\mathcal{B}$ equipped with norm $\norm{\cdot}$ and a closed and convex feasible set $\mathcal{X} \subseteq \cl{\dom \psi} \subseteq \mathcal{B}$. 
Suppose $\psi: \dom \psi \to \rset{}{}$ is a DGF, then the mirror map $\nabla \psi^*$ is $1/\Tilde{\mu}$-Lipschitz continuous, i.e., $\forall y_1, y_2 \in \mathcal{B}^*$, $\norm{\nabla \psi^*(y_1) - \nabla \psi^*(y_2)} \leq (1/\Tilde{\mu})\norm{y_1 - y_2}_*$. 
\end{lemma}
\begin{proof}
See \cite[Lemma~A.1]{huang2023zeroth}. 
\end{proof}
\Tblue{
To solve $\text{VI}(\mathcal{X}, F)$, the mirror descent can be expressed as: 
\begin{align}\label{eq:md}
X_{k+1} = P_{X_k, \mathcal{X}}(-\gamma_k g_k) = \nabla\psi^*(\nabla\psi(x_k) - \gamma_k g_k), 
\end{align}
where in the literature of stochastic VI, $g_k$ usually denotes some noise-corrupted first-order information queried at $X_k$ and $\gamma_k$ an appropriate updating step size. }
One prevalent assumption is that there exists a first-order oracle to generate $g_k$ after observing $X_k$, and given some proper filtration $(\mathcal{F}_k)_{k \in \nset{}{+}}$, it holds that $\expt{}{g_k \mid \mathcal{F}_k} = F(X_k)$ and $\expt{}{\norm{g_k}^2_* \mid \mathcal{F}_k}$ is a.s. bounded. 
The convergence properties of the actual sequences and the ergodic sequences have been extensively studied in \cite{mertikopoulos2019learning, mertikopoulos2018optimistic, juditsky2022unifying}. 

\section{BANDIT MIRROR DESCENT WITH FEEDBACK DELAYS}

\subsection{Residual Pseudo-Gradient Estimate}\label{subsec:rpg}
Our blanket assumption throughout is that the first-order oracle that returns $g_k$ is unavailable, and each player can only observe its realized objective value associated with the action taken. 
To address the absence of first-order information, we leverage a pseudo-gradient estimate called the residual pseudo-gradient estimate (RPG) \cite{huang2023zeroth} to approximate the missing information from the observed objective values. 
At each iteration $k$, initially, it is necessary to undertake the following perturbation step: 
\begin{align}\label{eq:perb}
\begin{split}
& \hat{X}^i_{k} = (1 - \frac{\delta_k}{r^i})X^i_{k}+\frac{\delta_k}{r^i}(p^i + r^iu^i_k) = \bar{X}^i_{k} + \delta_k u^i_k,
\end{split}
\end{align}
where 
$u^i_k$ is randomly sampled from the unit sphere in the $n^i-$dimensional Euclidean space and we define $u_k \coloneqq [u^i_k]_{i \in \mathcal{N}}$; 
$\delta_k$ represents the random query radius at iteration $k$;
$\mathbb{B}(p^i, r^i) \subseteq \mathcal{X}^i$ is an arbitrary fixed ball within the feasible set $\mathcal{X}^i$ that centers at $p^i$ with radius $r^i$; 
$\bar{X}^i_{k} \coloneqq (1 - \delta_k/r^i)X^i_{k} + (\delta_k/r^i) p^i$. 
The RPG associated with the states $X_{k}$ at $k$-th iteration leverages the realized objective values from the current iteration $\hat{J}^i_k \coloneqq J^i(\hat{X}^i_k; \hat{X}^{-i}_k)$ and the previous iteration $\hat{J}^i_{k-1} \coloneqq J^i(\hat{X}^i_{k-1}; \hat{X}^{-i}_{k-1})$, which is formally given by 
\begin{align}\label{eq:rpe}
G^i_k \coloneqq \frac{n^i}{\delta_k}(\hat{J}^i_k - \hat{J}^i_{k-1})u^i_k.
\end{align}

To analyze the properties of RPG, a smoothed version for each local objective function $J^i$ is leveraged:
\begin{align}
\Tilde{J}^i_{\delta}(x^i; x^{-i}) \coloneqq \frac{1}{\mathbb{V}^i_{\delta}} \int_{\delta \mathbb{S}_{-i}} \int_{\delta \mathbb{B}_i} J^i(x^i + \Tilde{\tau}^i; x^{-i} + \tau^{-i})d\Tilde{\tau}^i d\tau^{-i}, 
\end{align}
where $\mathbb{S}_{-i} \coloneqq \prod_{j \in \mathcal{N}^{-i}} \mathbb{S}_j \subseteq \rset{n^{-i}}{}$ with each $\mathbb{S}_j$ representing a unit sphere centered at the origin within $\rset{n^j}{}$; $\mathbb{B}_i$ denotes the unit ball centered at the origin inside $\rset{n^i}{}$; $\mathbb{V}^i_{\delta} \coloneqq \vol{\delta \mathbb{B}_i} \cdot \vol{\delta \mathbb{S}_{-i}}$ is the normalizing volume constant of the area that we are integrating over. 
One widely employed decomposition in the existing literature is that 
\begin{align*}
G^i_k =& \nabla_{x^i} J^i(X_{k}) + \big(G^i_k - \expt{}{G^i_k \mid \mathcal{F}_k}\big)  + \big(\expt{}{G^i_k \mid \mathcal{F}_k} - \nabla_{x^i} J^i(X_{k})\big),
\end{align*}
where we let $B^i_k \coloneqq \expt{}{G^i_k \mid \mathcal{F}_k} - \nabla_{x^i} J^i(X_{k})$ represent the systematic error and $V^i_k \coloneqq G^i_k - \expt{}{G^i_k \mid \mathcal{F}_k}$ the stochastic error. 
Denote $B_k \coloneqq [B^i_k]_{i \in \mathcal{N}}$ and $V_k \coloneqq [V^i_k]_{i \in \mathcal{N}}$. 
Let $(\mathcal{F}_k)_{k \in \nset{}{+}}$ be the filtration concerning the random exploration factor, i.e., $\mathcal{F}_k \coloneqq \sigma\{X_0, u_1, \ldots, u_{k-1}\}$. 
Then we have the following lemma to characterize the properties of $B_k$. 
\begin{lemma}\label{le:bias}
Suppose that Assumption~\ref{asp:objt-set} holds. 
Then at each iteration $k$, the conditional expectation satisfies $\expt{}{G^i_k \mid \mathcal{F}_k} = \nabla_{x^i} \Tilde{J}^i_{\delta_k}(\bar{X}_{k})$ a.s. for every $i \in \mathcal{N}$. 
Moreover, the systematic error $B_k$ possesses a decaying upper bound $\norm{B_k} \leq \alpha_B \delta_k$ for some positive constant $\alpha_B$.
\end{lemma}
\begin{proof}
See the proof of \cite[Lemma~1 \& Lemma~2]{huang2023zeroth}. 
\end{proof}

\subsection{Feedback Utilization Strategy}

The systematic error $B^i_k$ and stochastic error $V^i_k$ rooted in the estimate \eqref{eq:rpe} make it inappropriate to merely leverage the most recent first-order estimate multiple times until a more recent one arrives as what is done in \cite{huang2022convergence}; otherwise, the error will accumulate and endanger the convergence of the iterations. 
In view of this, we adopt the priority-based feedback utilization strategy: at each update, the first-order estimate with the earliest timestamp will be used and then discarded, similar to the approach employed in \cite{heliou20agradient}. 
\Tblue{However, the single-point estimate strategy used in \cite{heliou20agradient} mandates solely one realized function value, in which case it suffices to maintain a priority queue exclusively for these values. 
In contrast, the RPG adopted in this work requires two consecutive realized function values to obtain one estimate, which necessitates maintaining a cache to store observed function values and another priority queue for the resulting RPG estimates. 
}

\Tblue{In our feedback utilization strategy}, two information caches $\mathcal{P}^i_J$ and $\mathcal{P}^i_G$ are endowed for each player $i$. 
As reflected in \eqref{eq:rpe}, two subsequent objective values ($\hat{J}^i_k$ and $\hat{J}^i_{k-1}$) are a prerequisite to compute $G^i_k$, and it is possible that one arrives much earlier than the other. 
As such, cache $\mathcal{P}^i_J$ will store all the objective values received and pop out the ones that have been used twice in computing \eqref{eq:rpe}. 
For another thing, caused by the uncertainty in the feedback delay $d^i_k$, it is possible that at some iteration, player $i$ has no available first-order estimates, while for some other iterations, multiple estimates are at player $i$'s disposal. 
This motivates us to design $\mathcal{P}^i_G$ as a priority queue with the timestamp of each pseudo-gradient estimate as the key value. 
For notational convenience, we introduce a map $s^i: \nset{}{+} \to \nset{}{+}$ that maps from the current iteration to the iteration where the first-order estimate originates from. 
When \Tblue{$P^i_G$} is empty at iteration $k$, $s^i(k) = 1$ and the action remains unchanged. 
We also note that the map $s^i$ is implicitly parameterized by the random sample $\omega \in \Omega$ and could vary across this group of players under Assumption~\ref{asp:delay-plus} $(\romannum{1})$. 
To account for the heterogeneity in feedback delay $(d^i_k)_{i \in \mathcal{N}}$, we introduce a group iteration index map $s: \nset{}{+} \to \nset{N}{+}$, that projects from a certain iteration index $k$ to the stack of originated indices $[s^i(k)]_{i \in \playerN}$. 


Below, we present two lemmas that characterize the priority-based feedback utilization strategy, which our subsequent convergence analysis hinges upon. 
\ifarxiv
The proof is reported in Appendix~\ref{appd:delay-strategy}. 
\fi
\ifproceeding
The proof is reported in \cite[Appendix~A]{huang2023bandit}. 
\fi
\begin{lemma}\label{le:fdbk-prop}
For each player $i$ and arbitrary iteration $k \in \nset{}{+}$, we have the following: \\
$(\romannum{1})$ $K^i_{\varnothing}(k) \coloneqq \abs{\{s: \mathcal{P}^i_G = \varnothing, 1 \leq s \leq k\}} \leq \min\{k, \bar{d}(k) + 1\}$; \\ 
$(\romannum{2})$ if $s^i(k) \neq 1$, then $s^i(k) + \bar{d}(s^i(k)) \geq k$. 
\end{lemma}


\subsection{The MD Algorithm with Feedback Delays} 

\begin{algorithm}
\caption{Bandit Learning with Reward Delays of CPs Based on Mirror Descent (Player $i$)}\label{alg:md-delay}
\begin{algorithmic}[1]
\State \textbf{Initialize:} $X^i_1 \in \mathcal{X}^i \cap \dom{\psi^i}$ chosen arbitrary; 
take action $\hat{X}^i_1$ and $\hat{J}^i_1 = J^i(X^i_{1}; X^{-i}_{1})$ will arrive $\ceil{d^i_1}$ iterations later; 
$G^i_{1} = \boldsymbol{0}_{n^i}$; 
$p^i, r^i$ to be the center and radius of an arbitrary ball within the set $\mathcal{X}^i$
\Procedure{At the $k$-th iteration ($k \in \nset{}{+}$)}{}
\State Receive $\mathcal{R}^i_k \coloneqq \{(t, \hat{J}^i_t, u^i_t): k-1 < t + d^i_t \leq k\}$ 
\State $\mathcal{P}_J^i \leftarrow \mathcal{P}_J^i \cup \mathcal{R}^i_k$
\For{$(t, \hat{J}^i_t, u^i_t) \in \mathcal{R}^i_k$}
    \If{$(t+1, \hat{J}^i_{t+1}, u^i_{t+1}) \in \mathcal{P}^i_J$}
        \State $G^i_{t+1} \leftarrow \frac{n^i}{\delta_{t+1}}(\hat{J}^{i}_{t+1} - \hat{J}^{i}_{t})u^{i}_{t+1}$, $\mathcal{P}_G^i \coloneqq \mathcal{P}_G^i \cup \{G^i_{t+1}\}$
    \EndIf
    \If{$(t-1, \hat{J}^i_{t-1}, u^i_{t-1}) \in \mathcal{P}^i_J$}
        \State $G^i_{t} \leftarrow \frac{n^i}{\delta_t}(\hat{J}^{i}_t - \hat{J}^{i}_{t-1})u^{i}_t$, $\mathcal{P}_G^i \coloneqq \mathcal{P}_G^i \cup \{G^i_{t}\}$
    \EndIf
    \State $\mathcal{P}_J^i$ clears up the received feedback that has been utilized twice
\EndFor
\If{$\mathcal{P}_G^i \neq \varnothing$}
\State $s^i(k) \leftarrow \text{earliest index in }\mathcal{P}_G^i$, $\mathcal{P}_G^i \leftarrow \mathcal{P}_G^i \backslash \{ G^i_{s^i(k)} \}$
\Else
\State $s^i(k) \leftarrow 1$ \Comment{No update at this iteration}
\EndIf
\State $X^i_{k+1} \leftarrow P_{X^i_k, \mathcal{X}^i}(-\gamma_k G^i_{s^i(k)})$
\State Randomly sample the direction $u^i_{k+1}$ from $\mathbb{S}_i$
\State $\hat{X}^{i}_{k+1} \leftarrow (1 - \frac{\delta_{k+1}}{r^i})X^{i}_{k+1} + \frac{\delta_{k+1}}{r^i}(p^i + r^i u^{i}_{k+1})$ 
\State Take action $\hat{X}^{i}_{k+1}$ and the realized objective value $\hat{J}^{i}_{k+1} \coloneqq J^i(\hat{X}^{i}_{k+1}; \hat{X}^{-i}_{k+1})$ will arrive $\ceil{d^i_{k+1}}$ iterations later
\EndProcedure
\State \textbf{Return:} $\{\hat{X}^i_{k}\}_{i \in \playerN}$
\end{algorithmic}
\end{algorithm}

The fusion of MD, RPG, and the priority-based feedback utilization strategy results in the proposed algorithm for bandit learning in continuous games with feedback delays, which is detailed in Algorithm~\ref{alg:md-delay}. 
As has been discussed in \cite{huang2023zeroth}, one prominent benefit we can reap from RPG is that the associated stochastic error $V_k$ enjoys bounded variance if the decaying rate of step size is faster than that of query radius. 
It is worth mentioning that, Algorithm~\ref{alg:md-delay} leverages $\hat{G}_k = G_{s(k)} \coloneqq [G^i_{s^i(k)}]_{i \in \mathcal{N}}$ rather than $G_k$ to implement the action update at the $k$-th iteration, which is susceptible to the approximation errors stemming from bandit estimation and feedback delays. 
The existence of feedback delays then disrupts the recurrent relation characterizing $(\hat{G}_k)_{k \in \nset{}{+}}$, as a result of which, the analysis of the boundedness of the stochastic error and the estimates $G^i_k$ in \cite{huang2023zeroth} cannot be directly carried over. 
To facilitate later analysis, we set $\hat{G}_1 = G_1 = G_0 = \boldsymbol{0}_{n}$ and  $\hat{J}^i_0 = \hat{J}^i_1$. 
In the lemma below, we will present the sufficient condition to guarantee that the estimates $\hat{G}_k$ enjoy a uniform upper bound across $k \in \nset{}{+}$ and $\omega \in \Omega$. 
\ifarxiv
The proof is reported in Appendix~\ref{appd:bounded-rpg}. 
\fi
\ifproceeding
The proof is reported in \cite[Appendix~B]{huang2023bandit}. 
\fi

\begin{lemma}\label{le:bounded-rpg}
Suppose that Assumptions~\ref{asp:objt-set} and \ref{asp:delay} hold. 
Moreover, step size $(\gamma_k)_{k \in \nset{}{+}}$ and query radius $(\delta_k)_{k \in \nset{}{+}}$ are monotonically decreasing and satisfy: 
$\lim_{k\to\infty}\gamma_k = 0$, $\sum_{k \in \nset{}{+}} \gamma_k = \infty$, $\lim_{k\to\infty}\delta_k = 0$, $\delta_k/\delta_{k+1}$ is uniformly bounded for all $k \in \nset{}{+}$, $\lim_{k \to \infty} \gamma_k/\delta_k = 0$. 
Considering $(\hat{G}_{k})_{k \in \nset{}{+}}$ generated by Algorithm~\ref{alg:md-delay}, we have $\sup_{k \in \nset{}{+}} \norm{\hat{G}_{k}}_* < \infty$. 
\end{lemma}

For the feedback-delay scenario, the randomness originates from two sources: the random exploration factor at each iteration $u^i_k$ and the feedback delay $d^i_k$ associated with the realized objective value $\hat{J}^i_k$. 
Let the $\sigma$-field reflecting the delay information up to iteration $k$ be denoted as:
\begin{align}
\mathcal{F}^d_k \coloneqq \sigma\{d^i_t: \forall i \in \mathcal{N}, 1\leq t\leq k\}  
\end{align}
Note that $s^i(t) \in \mathcal{F}^d_k$ for all $1 \leq t \leq k$ and the available information respecting random exploration factors $u^i_t$ depends on $\mathcal{F}^d_k$. 
Based on the observation, we are prompted to consider a more suitable $\sigma$-field $\Tilde{\mathcal{F}}_{s(k)}$ for this specific problem, rather than the $\sigma$-field $\mathcal{F}_{k}$ previously discussed in Sec.~\ref{subsec:rpg}, which is defined as: 
\begin{align}
 \Tilde{\mathcal{F}}_k \coloneqq \sigma\big(\mathcal{F}^d_k \cup \{u^i_{s^i(t)}: \forall i \in \mathcal{N}, 1 \leq t \leq k-1\}\big). 
\end{align} 
With this definition in hand, we can then proceed to discuss the asymptotic convergence results for the actual sequence of play generated by Algorithm~\ref{alg:md-delay}. 
\ifarxiv
The proof can be found in Appendix~\ref{appd:convg}. 
\fi
\ifproceeding
The proof can be found in \cite[Appendix~C]{huang2023bandit}. 
\fi
\begin{theorem}\label{thm:as-convg}
Suppose the game $\mathcal{G}$ under consideration satisfies Assumptions~\ref{asp:objt-set} to \ref{asp:recip} and all the players of $\mathcal{G}$ follow Algorithm~\ref{alg:md-delay} throughout the process. 
Moreover, the step size $(\gamma_k)_{k \in \nset{}{+}}$ and the query radius $(\delta_k)_{k \in \nset{}{+}}$ are chosen as $\gamma_k = \gamma_0/(k + K_{\gamma})^{\alpha_\gamma}$ and $\delta_k = \delta_0/(k + K_{\delta})^{\alpha_\delta}$, respectively. 
The selected parameters satisfy 
$
0.5 < \alpha_\gamma \leq 1, \alpha_{\gamma} > \alpha_{\delta}, \alpha_{\gamma} + \alpha_{\delta} > 1, 2\alpha_{\gamma} - \alpha_d > 1.
$
Then the actual sequence of play $(\hat{X}_k)_{k \in \nset{}{+}}$ converges to one of the CP $x_*$ almost surely. 
\end{theorem}








\section{Numerical Experiments}

To illustrate the effectiveness of the proposed algorithm, we provide a numerical example of the thermal load management problem in a building complex. 
Suppose the load aggregator under study consisting of $N$ buildings, indexed by $\mathcal{N} \coloneqq \{1, \ldots, N\}$. 
Over a given time horizon $\mathcal{T} \coloneqq \{1, \ldots, T\}$, we use $x^i_t$ to represent the power consumption of building $i$ at a certain time slot $t \in \mathcal{T}$.
Moreover, the concatenations $x^i \coloneqq [x^i_t]$ and $x \coloneqq [x^i]$ denote the power profile of building $i$ for all time slots and the energy profile of all buildings in this load aggregator, respectively. 
The internal pricing mechanism under consideration \cite{jiang2021game} discourages peak-demand usage by incorporating an approximate version of Shapley value, where each building $i$'s share of peak demand is defined as 
$R^i(x) = \sum_{\mathcal{C}_j: i \in \mathcal{C}_j} \frac{(N - \abs{\mathcal{C}_j})!(\abs{\mathcal{C}_j} - 1)!}{N!}\Big(V(\mathcal{C}_j, x) - V(\mathcal{C}_j\backslash \{i\}, x)\Big),
$
where $\mathcal{C} \coloneqq \{\mathcal{C}_{1}, \ldots \mathcal{C}_{n_c}\}$ with each $\mathcal{C}_j \subseteq \mathcal{N} (j = 1, \ldots, n_c)$ denotes the clique set; 
the function $V$ is defined as $V(\mathcal{C}_j, x) = \frac{1}{C}\log \Big( \sum_{t \in \mathcal{T}} \exp\big(\sum_{l \in \mathcal{C}_j} C x^l_{t}\big) \Big)$, \Tblue{where $C \in \rset{}{++}$ is a constant sufficiently large to make the log-sum-exp function a proper smooth approximation to the maximum function. }

With knowledge of the power profile $x^{-i}$ of other buildings, each building $i$ seeks to find an optimal power control strategy, which can be expressed as follows: 
\begin{align}
\begin{split}
& \minimize_{x^i \in \mathcal{X}^i} \;(p_e)^Tx^i + Q^i(x^i) + p_d \cdot R^i(x)  \\
& \subj \;r^i_{t} = a^i r^i_{t-1} + b^i x^i_{t}, \; y^i_{t} = c^i r^i_{t},  \\
& \qquad \qquad  \ubar{y}^i_{t} \leq y^i_{t} \leq \bar{y}^i_{t}, \; 0 \leq x^i_{t} \leq \bar{x}^{i}, \forall t \in \mathcal{T}, 
\end{split}
\end{align}
where $p_e \in \rset{T}{++}$ denotes the energy price and $p_d \in \rset{}{++}$ penalized the peak electricity usage of the aggregator; 
a strongly convex quadratic function $Q^i$ is introduced for the convergence purpose; 
$y^i_t$ denotes the temperature of building $i$ at the $t$-th time slot and its dynamics are characterized by the first and second equality constraints; 
the third constraint enforces that the temperature $y^i_t$ should be within a comfort zone $[\ubar{y}^i_{t}, \bar{y}^i_{t}]$; 
the last constraint reflects the system power capacity for each building. 
It can be proved that this multi-player game admits a potential function 
$\Phi(x) = \sum_{i \in \playerN} \Big((p_e)^Tx^i + Q^i(x^i)\Big) + p_d \cdot \sum_{\mathcal{C}_j \in \mathcal{C}} \frac{(N - \abs{\mathcal{C}_j})!(\abs{\mathcal{C}_j} - 1)!}{N!}V(\mathcal{C}_j, x)$. 

In the experiments, twenty buildings $(N=20)$ are involved in this game, and each building needs to determine its energy profile for four different time slots $(T=4)$. 
Suppose there are six cliques and the number of buildings within each clique ranges from three to eight. 
For $Q^i(x^i) = (x^i)^T\diag{\lambda_{i1}, \ldots, \lambda_{in^i}}x^i$, each diagonal entry $\lambda_{ij}$ is randomly sampled from $[0.04, 0.06]$. 
The query radius $\delta_k$ and the step size $\gamma_k$ are set to be $\delta_k = 1/(k+10)^{0.6}$ and $\gamma_k = 1/(k+10^3)^{0.9}$, respectively. 
Regarding the feedback delay $d^i_k$, we consider the case when $d^i_k$ is upper bounded by $\bar{d}_k = 10^3$ while the realized values of $d^i_k$ vary across different buildings. 
In addition, several experiments are conducted under the setup that $d^i_k$ is homogeneous in this group of buildings and grows sublinearly. 
To compare with the existing work, we implement the method in \cite{heliou20agradient} with $\delta_k = 1/(k+10)^{0.35}$ and $\gamma_k = 1/(k+10^3)^{0.9}$ as required by the associated convergence theorem. 
Two metrics are employed to measure the performance of Algorithm~\ref{alg:md-delay}, which include the relative distance between the NE and the perturbed actions, $\norm{\hat{X}_{k} - x_*}_2 / \norm{x_*}_2$, and the difference between the potential function's optimal value and the values at the perturbed actions, $\Phi(\hat{X}_{k}) - \Phi_*$. 

The numerical results are illustrated in Fig~\ref{fig:thm_ctrl4}. 
It can be observed that when the feedback delay $d^i_k$ grows no faster than $O(\sqrt{k})$, the convergence rates of the generated sequences are dominated by the first-order estimation error and no significant difference is noted among $\bar{d}_k = 10^3$, $d_k = k^{0.1}$, $d_k = 5 k^{0.5}$, and $d_k = 10 k^{0.5}$. 
When the delay time $d^i_k$ grows faster and even approaches the rate of $O(k)$, the errors induced by the feedback delay outweigh those induced by the estimation error, as reflected in the curves associated with $d_k = 5 k^{0.75}$ and $d_k = 5 k^{0.99}$. 
Furthermore, the results in Fig.~\ref{fig:thm_ctrl4} indicate that Algorithm~\ref{alg:md-delay} exhibits reduced variance, more consistent sequences of play, and faster convergence compared to the existing method in \cite{heliou20agradient}.

\begin{figure}
    \centering
    \includegraphics[width=0.45\textwidth]{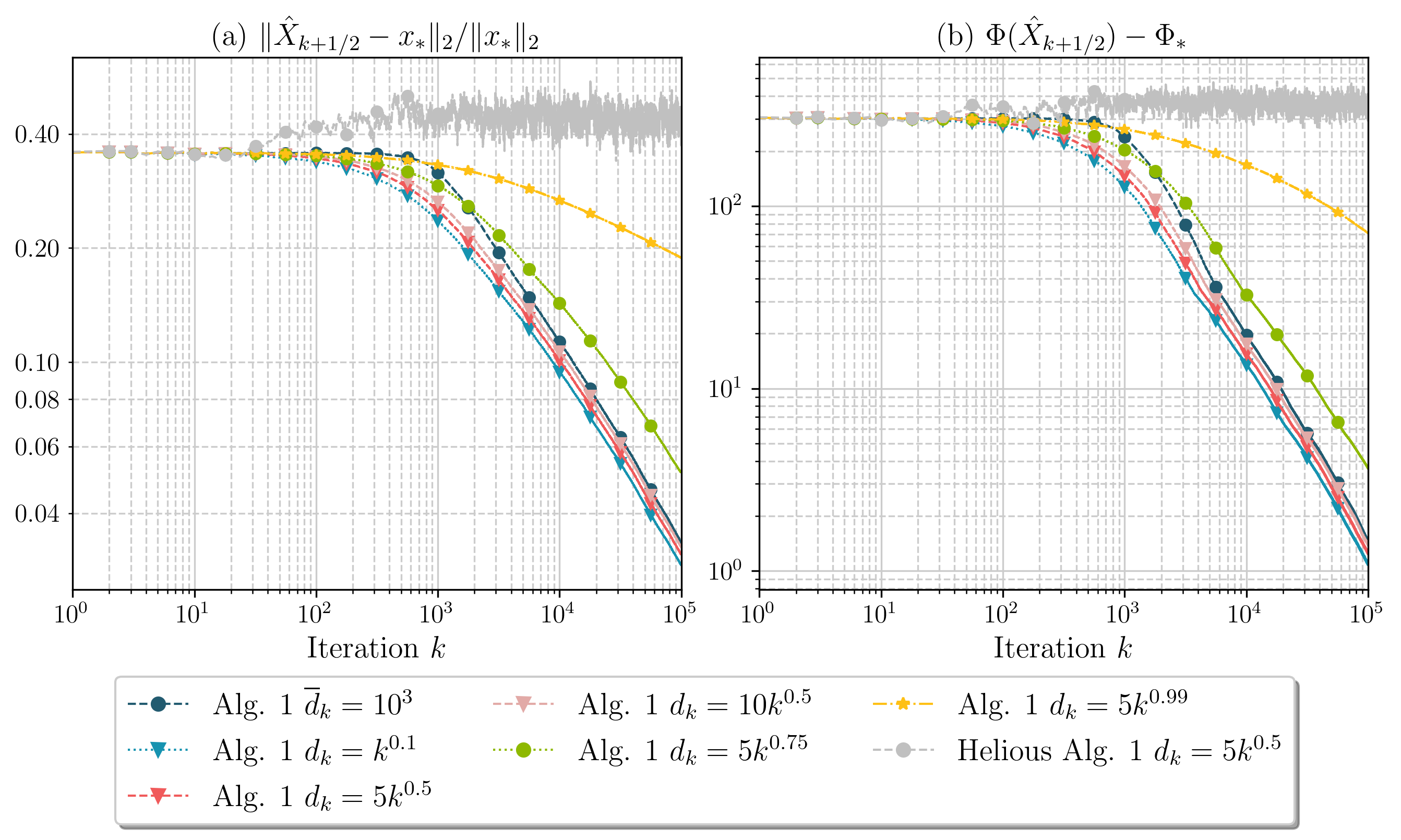}
    \caption{Performance of the Proposed Algorithm Confronted with Homogeneous and Heterogeneous Feedback Delays}
    \label{fig:thm_ctrl4}
\end{figure}



\section{Conclusion}
This paper studies the problem of bandit learning in multi-player continuous games, which is further complicated by information delays. 
Compared with the existing method introduced in \cite{heliou20agradient}, the algorithm proposed in this paper incorporates the residual pseudo-gradient estimation strategy and the mirror descent iteration, which loosens the conditions imposed upon the query radius and the step sizes. 
The a.s. convergence of the actual sequences of play generated by the proposed algorithm is established for pseudo-monotone plus games. 
One important direction for future research concerns the case where the feedback delays grow as the iteration proceeds and at the same time, they are heterogeneous across the participants. 
Another potential future direction resides in designing an algorithm that could tackle a more general class of multi-player games, such as merely monotone games, which are prevalent in the modeling of practical problems. 
Nevertheless, when applied to merely monotone games, mirror descent and most of its variants fail to converge and are prone to be trapped in spurious solutions.  

\ifarxiv
\appendices 
\input{appendix}
\fi

\bibliographystyle{IEEEtran}
\bibliography{IEEEabrv,references}

\end{document}

%% file: macros.tex
\newcommand{\abs}[1]{\lvert#1\rvert}%
\newcommand{\norm}[1]{\lVert#1\rVert}

\newcommand{\Bnorm}[1]{\Big\lVert#1\Big\rVert}

\newcommand{\diag}[1]{\text{diag}(#1)}
\newcommand{\cl}[1]{\text{cl}(#1)}
\newcommand{\expt}[2]{\mathbb{E}_{#1}[#2]}

\DeclarePairedDelimiter\ceil{\lceil}{\rceil} 

\newcommand\ubar[1]{\stackunder[1.2pt]{$#1$}{\rule{.8ex}{.075ex}}}

\DeclareMathOperator{\dom}{dom}

\DeclareMathOperator{\minimize}{minimize}

\DeclareMathOperator{\subj}{subject\:to}

\newcommand{\playerN}{\mathcal{N}}

\newcommand{\rset}[2]{\mathbb{R}^{#1}_{#2}}

\newcommand{\nset}[2]{\mathbb{N}^{#1}_{#2}}

\newcommand{\vol}[1]{\text{vol}(#1)}

\DeclareMathOperator{\argmin}{argmin}
\DeclareMathOperator{\argmax}{argmax}

\usepackage{accents}

\newcommand{\Tblue}[1]{\textcolor{black}{#1}}

%% file: appendix.tex
\section*{Appendix}
\addcontentsline{toc}{section}{Appendix}
\renewcommand{\thesubsection}{\Alph{subsection}}

\newtheorem{appdxlemma}{Lemma}
\numberwithin{appdxlemma}{subsection} 
\numberwithin{equation}{subsection} 

\subsection{Properties of the Feedback Utilization Strategy}\label{appd:delay-strategy}
\begin{proof}(Proof of Lemma~\ref{le:fdbk-prop})
For claim $(\romannum{1})$, our focus will be proving that $K^i_{\varnothing}(k) \leq \bar{d}(k) + 1$ when $\bar{d}(k) + 1 < k$, since it is evident that $K^i_{\varnothing}(k) \leq k$. 
For a fixed $k \in \nset{}{+}$, we denote the constant $D_k \coloneqq \bar{d}(k)$, and the delay times up to the iteration $k$ satisfy $d^i_t \leq D_k$ for $t=1, \ldots, k$. 
If each realized $\hat{J}^i_t$ arrives $D_k$ iterations later, then throughout the iterate, $(\hat{J}^i_t)_{t=1, \cdots, k - D_k}$ will be received in sequence. 
At the $t$-th iteration with $D_k + 2 \leq t \leq k$, player $i$ will have access to the estimate $G^i_{t-D_k}$ and employ it in the action update. 
The total count of iterations without action update equals $D_k + 1$.
Now we return to the case that the delay time is characterized by the random variable $d^i_t$. 
Since $d^i_t \leq D_k$, the estimates $(G^i_{t})_{t=2, \ldots, k-D_k}$ will be available no later than the constant case above and hence will be used in the action update, which further implies that $K^i_{\varnothing}(k) \leq \bar{d}(k) + 1$. 

For claim $(\romannum{2})$, we will prove it by induction. 
Before proceeding to the analysis, we make some notational conventions regarding the iteration indices and the sequence $(s^i(k))_{k \in \nset{}{+}}$. 
Recall that when the estimate cache $\mathcal{P}^i_G$ is empty at iteration $k$, $s^i(k)$ is manually set to $1$. 
We let $(\ell_m)_{m \in \nset{}{+}} \coloneqq (k: s^i(k) \neq 1)_{k \in \nset{}{+}}$ to denote the iteration indicies with action update. 
In addition, we will use $\mathcal{P}^i_G(k)$ to represent the temporary state of the cache $\mathcal{P}^i_G$ at line 14 of the $k$-th iteration and $t \in \mathcal{P}^i_G(k)$ to indicate $G^i_t \in \mathcal{P}^i_G(k)$. 

\textbf{Initial condition: }For $\ell_1$ and $G^i_{s^i(\ell_1)}$, we have either $s^i(\ell_1) + \bar{d}(s^i(\ell_1)) \geq \ell_1$ or $s^i(\ell_1)-1 + \bar{d}(s^i(\ell_1)-1) \geq \ell_1$; 
otherwise, $G^i_{s^i(\ell_1)}$ can be evaluated and consumed at an earlier iteration, since it is the first estimate and there is no queuing issue in $\mathcal{P}^i_G$. 

\textbf{Induction step: }For an arbitrary $k \in \nset{}{+}$, we assume that $s^i(\ell_k) + \bar{d}(s^i(\ell_k)) \geq \ell_k$ and need to show that the statement hold for $k+1$, i.e., $s^i(\ell_{k+1}) + \bar{d}(s^i(\ell_{k+1})) \geq \ell_{k+1}$. \\
\textbf{Case I ($s^i(\ell_{k+1}) \in \mathcal{P}^i_{G}(\ell_k)$): }
In this case, the first observation is that $s^i(\ell_{k+1}) \geq s^i(\ell_{k}) + 1$ since $\mathcal{P}^i_{G}$ pops out the estimate with the earliest timestamp and $s^i(\ell_k)$ is used first. 
Moreover, given that $G^i_{\ell_{k+1}}$ is already available at the $\ell_k$-th iteration, we have $\ell_{k+1} = \ell_k + 1$. 
Altogether, $s^i(\ell_{k+1}) + \bar{d}(s^i(\ell_{k+1})) \geq s^i(\ell_{k}) + 1 + \bar{d}(s^i(\ell_{k})) \geq \ell_k + 1 = \ell_{k+1}$. \\
\textbf{Case II ($s^i(\ell_{k+1}) \notin \mathcal{P}^i_{G}(\ell_k)$): } In the case where $s^i(\ell_{k+1}) \notin \mathcal{P}^i_{G}(\ell_k)$, it can be possible that $\abs{\mathcal{P}^i_G(\ell_k)} = 1$ or $\abs{\mathcal{P}^i_G(\ell_k)} \geq 2$ but $G^i_{s^i(\ell_{k+1})}$ becomes available at $\ell_k + 1$ and has the earliest timestamp among all available estimates in $\mathcal{P}^i_{G}(\ell_{k+1})$. 
In either case, we must have either $s^i(\ell_{k+1}) + \bar{d}(s^i(\ell_{k+1})) \geq \ell_{k+1}$ or $s^i(\ell_{k+1})-1 + \bar{d}(s^i(\ell_{k+1})-1) \geq \ell_{k+1}$ in the similar vein of the initial condition; otherwise, $G^i_{s^i(\ell_{k+1})}$ can be evaluated and consumed at an earlier stage or $s^i(\ell_{k+1}) \in \mathcal{P}^i_{G}(\ell_k)$. 
\end{proof}

\subsection{Proof of Lemma~\ref{le:bounded-rpg}}\label{appd:bounded-rpg}
\begin{proof}
We start by deriving a recurrent inequality for the sequence $(\hat{G}_k)_{k \in \nset{}{+}} = (G_{s(k)})_{k \in \nset{}{+}}$. 
For the segment corresponding to player $i$, we notice that
\begin{align*}
& \norm{\hat{G}^i_k}_* = \norm{G^i_{s^i(k)}}_* \leq \Bnorm{\frac{n^i}{\delta_{s^i(k)}}(\hat{J}^i_{s^i(k)} - \hat{J}^i_{s^i(k)-1})u^i_{s^i(k)}}_* \\
& \overset{(a)}{\leq} \frac{n^i}{\delta_{s^i(k)}} \abs{\langle \nabla_{x}J^i(\Tilde{X}), \hat{X}_{s^i(k)} - \hat{X}_{s^i(k)-1}\rangle}\norm{u^i_{s^i(k)}}_* \\
& \overset{(b)}{\leq} \frac{n^i}{\delta_{s^i(k)}} \bar{\nabla}_i \cdot \norm{\hat{X}_{s^i(k)} - \hat{X}_{s^i(k)-1}} \cdot \bar{u}^i_*, 
\end{align*}
where $(a)$ follows from the mean value theorem, and $\Tilde{X}$ denotes some convex combination of $\hat{X}_{s^i(k)}$ and $\hat{X}_{s^i(k)-1}$; 
in $(b)$, we take the maximum $\bar{\nabla}_i \coloneqq \max_{x \in \mathcal{X}} \norm{\nabla_x J^i(x)}_*$ and denote the constant $\bar{u}^i_* \coloneqq \norm{u}_*$ where $\norm{u}_2 = 1$. 
Based on this, we next derive a bound for $\norm{\hat{X}_{s^i(k)} - \hat{X}_{s^i(k)-1}}$ as follows
\begin{align*}
& \norm{\hat{X}_{s^i(k)} - \hat{X}_{s^i(k)-1}} \overset{(a)}{=} \norm{X_{s^i(k)} - X_{s^i(k)-1} + \delta_{s^i(k)}\varphi_{s^i(k)} \\
& \qquad - \delta_{s^i(k)-1}\varphi_{s^i(k)-1}} \overset{(b)}{\leq}  \norm{X_{s^i(k)} - X_{s^i(k)-1}} + \delta_{s^i(k)}\Bar{\varphi}, 
\end{align*}
where we let $\varphi_{s^i(k)} \coloneqq R^{-1}(p - X_{s^i(k)} + Ru_{s^i(k)})$ in $(a)$; 
for $(b)$, we can find a constant $\bar{\varphi}$ such that $\bar{\varphi} \geq \norm{\varphi_{s^i(k)} - (\delta_{s^i(k)-1}/\delta_{s^i(k)})\varphi_{s^i(k)-1}}$ given that $\varphi_{s^i(k)}$ resides inside a bounded set and the ratio $\delta_{s^i(k)-1}/\delta_{s^i(k)}$ is uniformly upper bounded by some constant. 
By applying the MD iterate and the $1/\Tilde{\mu}$-Lipschitz continuity of the mirror map from Lemma~\ref{le:md-lips}, we have 
\begin{align*}
& \norm{X_{s^i(k)} - X_{s^i(k)-1}}  = \norm{\nabla\psi^*(\nabla \psi(X_{s^i(k)-1}) - \gamma_{s^i(k)-1}\hat{G}_{s^i(k)-1})\\
& \qquad 
 - \nabla \psi^*(\nabla \psi(X_{s^i(k)-1}))}  \leq \frac{\gamma_{s^i(k)-1}}{\Tilde{\mu}}\norm{\hat{G}_{s^i(k)-1}}_*. 
\end{align*}
Thus, the pseudo-gradient of the game can be characterized by the following relation:
\begin{align*}
\norm{\hat{G}_k}_* & \leq \sum_{i \in \mathcal{N}} \norm{\hat{G}^i_k}_* \leq \sum_{i \in \mathcal{N}}\Big(\frac{\gamma_{s^i(k)-1}}{\delta_{s^i(k)}}\frac{n^i\bar{\nabla}_i\bar{u}^i_*}{\Tilde{\mu}}\norm{\hat{G}_{s^i(k)-1}}_*\Big)+\beta_1 \bar{\varphi}  \\
& \leq \frac{\beta_1}{\Tilde{\mu}}\sum_{i \in \mathcal{N}}\Big(\frac{\gamma_{s^i(k)-1}}{\delta_{s^i(k)}}\norm{\hat{G}_{s^i(k)-1}}_*\Big) + \beta_1 \bar{\varphi}, 
\end{align*}
where $\beta_1 \coloneqq \sum_{i \in \mathcal{N}} n^i\bar{\nabla}_i \bar{u}^i_*$. 
For an arbitrary random sample $\omega \in \Omega$, we can obtain the following deterministic inequality:
\begin{align*}
\norm{\hat{G}_k}_*(\omega) \leq \frac{\beta_1}{\Tilde{\mu}}\sum_{i \in \mathcal{N}}\Big(\frac{\gamma_{s^i(k)-1}}{\delta_{s^i(k)}}\norm{\hat{G}_{s^i(k)-1}}_*(\omega)\Big) + \beta_1 \bar{\varphi}. 
\end{align*}
Lemma~\ref{le:fdbk-prop} suggests that $s^i(k) \geq k - \bar{d}(s^i(k)) \geq k - \bar{d}(k)$, and we define the map $\pi_{\omega}: \nset{}{+} \to \nset{}{+}$ parameterized by $\omega \in \Omega$ as: 
\begin{align*}
\pi_{\omega}(k) = \argmax_{1\vee(k - \bar{d}(k))\leq t\leq k-1, t \in \nset{}{+}} \frac{\gamma_{t}}{\delta_{t+1}}\norm{\hat{G}_{t}}_*(\omega). 
\end{align*}
By definition, $\pi_{\omega}(k) < k$, $\forall \omega$ and $k$. 
With the introduction of $\pi_{\omega}$, we can tackle the heterogeneity in $s^i(k)$ and obtain: 
\begin{align*}
\norm{\hat{G}_k}_*(\omega) & \leq \frac{\beta_1}{\Tilde{\mu}}\sum_{i \in \mathcal{N}}\big(\frac{\gamma_{\pi_{\omega}(k)}}{\delta_{\pi_{\omega}(k)+1}}\norm{\hat{G}_{\pi_{\omega}(k)}}_*(\omega)\big) + \beta_1 \bar{\varphi} \\
& = \beta_2(\pi_{\omega}(k))\norm{\hat{G}_{\pi_{\omega}(k)}}_*(\omega) + \beta_1 \bar{\varphi}, 
\end{align*}
where we let $\beta_2(t) \coloneqq\frac{\beta_1N}{\Tilde{\mu}}\cdot \frac{\gamma_{t}}{\delta_{t+1}}$. 
Observe that as $k \to \infty$, it follows that $k - \bar{d}_k \to \infty$, which further implies that $\beta_2(\pi_{\omega}(k)) \to 0$. 
Let $\bar{\beta}_2(k) \coloneqq \beta_2(k) \vee 1$, and we can recursively construct a constant $g_\star$ that could serve as a worst-case upper bound regardless of $\omega$ as 
\begin{align*}
g_\star = \beta_1\bar{\varphi} (1 + \sum_{\ubar{T}=1}^{K_\star-1}\prod_{t=\ubar{T}}^{K_\star-1}\bar{\beta}_2(t)), 
\end{align*}
\Tblue{where} we can find a constant index $K_\star$ independent of $\omega$, such that $\beta_2(k) < \varepsilon$ for some $\varepsilon < 1$ and all $k \geq K_\star$. 
For another thing, for an arbitrary $k \in \nset{}{+}$, $\norm{\hat{G}_k}_*(\omega)$ can be recurrently upper bounded regarding the sequence
$(\norm{\hat{G}_k}_*, \norm{\hat{G}_{\pi_{\omega}(k)}}_*, \norm{\hat{G}_{(\pi_{\omega})^2(k)}}_*, \ldots, \norm{\hat{G}_{1}}_*)$, where $\omega$ is omitted for brevity. 
If $\pi_{\omega}(k) < K_{\star}$, there will be less than $k-1$ recurrent inequalities to link $\norm{\hat{G}_k}_*$ back to $\norm{\hat{G}_1}_* = 0$. 
As such, the constant $g_\star$ serves as a uniform upper bound for all the $\norm{\hat{G}_k}_*$ with $\pi_{\omega}(k) < K_{\star}$. 
In the case where $\pi_{\omega}(k) \geq K_{\star}$, we focus on the latter portion of the estimate sequence, i.e., $(\norm{\hat{G}_k}_*, \norm{\hat{G}_{\pi_{\omega}(k)}}_*, \ldots, \norm{\hat{G}_{K_\Delta}}_*, \norm{\hat{G}_{\pi_{\omega}(K_\Delta)}}_*)$ with $\pi_{\omega}(K_\Delta) < K_\star \leq K_\Delta$ and $\norm{\hat{G}_{\pi_{\omega}(K_\Delta)}}_*(\omega) \leq g_\star$. 
For this subsequence, we have $\norm{\hat{G}_t}_*(\omega) \leq \varepsilon \norm{\hat{G}_{\pi_{\omega}(t)}}_*(\omega) + \beta_1\bar{\varphi}$, which gives us a stable linear discrete-time system with $\varepsilon < 1$. 
Thus, there exists a constant $\bar{g}_\star$ such that $\sup_{k \in \nset{}{+}, \omega \in \Omega} \norm{\hat{G}_k}_*(\omega) \leq \bar{g}_\star$. 
\end{proof}

\subsection{Almost-Sure Convergence of the Proposed Algorithm}\label{appd:convg}

To facilitate our later discussion, denote the event $E^i_k \coloneqq \{\mathcal{P}^i_{G} \neq \varnothing \text{\;at iteration\;}k\}$ and notice that $E^i_k \in \Tilde{\mathcal{F}}_k$. 
In addition, let $E_k \coloneqq \cap_{i \in \mathcal{N}}E^i_k$. 

\begin{appdxlemma}\label{le:insum-gamma}
Suppose that step size $\gamma_k = \gamma_0/(k + K_{\gamma})^{\alpha_\gamma}$ with $\alpha_\gamma \leq 1$, and Assumption~\ref{asp:delay} holds. Then, $\sum_{k \in \nset{}{}} \gamma_k \mathds{1}_{E_k}(\omega) = \infty$ for all $\omega \in \Omega$. 
\end{appdxlemma}
\begin{proof}
For arbitrary $\omega \in \Omega$, we have 
\begin{align*}
& \sum_{k\in \nset{}{+}} \gamma_k \mathds{1}_{E_k}(\omega) = \lim_{K \to \infty} \sum_{k=1}^{K} \gamma_k \mathds{1}_{E_k}(\omega) \overset{(a)}{\geq} \lim_{K \to \infty}\sum^K_{K \wedge \ceil{\bar{d}(K) + 2}N}\gamma_k \\
& = \lim_{K \to \infty}\sum^K_{\ceil{\bar{d}(K) + 2}N}\gamma_k \quad \cdots \cdots \quad (\star), 
\end{align*}
where $(a)$ is a result of Lemma~\ref{le:fdbk-prop}, i.e., from the perspective of each player, for the first $K$ iterations, there are at most $\bar{d}(K)+1$ iterations without action update. 
On account of the monotonically decreasing property of $\gamma_k$, the worst-case scenario is that these $\bar{d}(K)+1$ iterations sit at the very beginning of the process and are different across this group of players, contributing to a factor of $N$. 
When $\alpha_\gamma < 1$, $(\star) \geq \lim_{K \to \infty} \int_{\ceil{\bar{d}(K) + 2}N}^{K}\gamma_0(s + K_{\gamma})^{-\alpha_\gamma}ds = \lim_{K \to \infty}\frac{\gamma_0}{1 - \alpha_{\gamma}}[(s + K_{\gamma})^{1 - \alpha_{\gamma}}]^K_{\ceil{\bar{d}(K) + 2}N} = \infty$ since $1 - \alpha_{\gamma} > 0$ and $\Bar{d}(K) \propto K^{\alpha_d}$ with $\alpha_d < 1$. 
Likewise, when $\alpha_\gamma = 1$, $(\star) \geq \lim_{K \to \infty} \int_{\ceil{\bar{d}(K) + 2}N}^{K}\gamma_0(s + K_{\gamma})^{-1}ds = \lim_{K \to \infty}\gamma_0[\log(s + K_{\gamma})]^K_{\ceil{\bar{d}(K) + 2}N} = \infty$ given $\alpha_d < 1$.   
\end{proof}

In light of this lemma, we can now proceed to prove our main result, which establishes the a.s. convergence of the proposed algorithm. 

\begin{proof}(Proof of Theorem~\ref{thm:as-convg})
By applying the standing inequality of mirror descent (\cite[Lemma~A.2]{huang2023zeroth}), for an arbitrary CP $x_* \in \mathcal{X}$, we have:
\begin{align}\label{eq:std-ineq}
D(x_*, X_{k+1}) \leq D(x_*, X_k) - \gamma_k \langle \hat{G}_k, X_k - x_*\rangle + \frac{\gamma_k^2}{2\Tilde{\mu}}\norm{\hat{G}_k}_*^2. 
\end{align}
By the fact that $\sum_{k \in \nset{}{+}}\gamma_k^2 < \infty$ from the assumption and $\sup_{k \in \nset{}{+}, \omega \in \Omega} \norm{\hat{G}_k}_* \leq \bar{g}_{\star}$ from Lemma~\ref{le:bounded-rpg}, we can claim that $\sum_{k\in \nset{}{+}} \gamma_k^2/2\Tilde{\mu}\norm{\hat{G}_k}_*^2 < \infty$, i.e., this part will play a comparatively negligible role in the convergence analysis. 
If $E^i_k$ happens, $\langle\hat{G}^i_k, X^i_k - x^i_*\rangle$ can be decomposed as
\begin{align*}
& \langle\hat{G}^i_k, X^i_k - x^i_*\rangle = \langle \nabla_{x^i}J^i(X_{k}), X^i_k - x^i_*\rangle + \sum_{t=s^i(k) + 1}^{k}\langle \nabla_{x^i}J^i(X_{t-1}) \\
& \quad - \nabla_{x^i}J^i(X_{t}) , X^i_k - x^i_*\rangle + \langle B^i_{s^i(k)} + V^i_{s^i(k)}, X^i_k - x_*\rangle
\end{align*}
The stacked inner product in \eqref{eq:std-ineq} can then be examined individually and be decomposed as follows:
\begin{align*}
& -\langle \hat{G}_k, X_k - x_*\rangle = -\langle F(X_k), X_k - x_*\rangle\mathds{1}_{E_k} - \sum_{i \in \mathcal{N}} \mathds{1}_{(E_k)^c\cap E^i_k}\cdot \\
& \quad  \langle \nabla_{x^i} J^i(X_k), X^i_k - x^i_*\rangle - \sum_{i \in \mathcal{N}} \Big(\sum_{t=s^i(k) + 1}^{k}\langle \nabla_{x^i}J^i(X_{t-1}) - \nabla_{x^i}J^i(X_{t}), \\
& \qquad X^i_k - x^i_*\rangle + \langle B^i_{s^i(k)} + V^i_{s^i(k)}, X^i_k - x_*\rangle\Big)\mathds{1}_{E^i_k}\\
& \overset{(a)}{\leq} -\langle F(X_k), X_k - x_*\rangle\mathds{1}_{E_k} - \sum_{i \in \mathcal{N}} \mathds{1}_{(E_k)^c\cap E^i_k}\langle \nabla_{x^i} J^i(X_k), X^i_k - x^i_*\rangle\\
& + \sum_{i \in \mathcal{N}}\Big(\sum_{t=s^i(k) + 1}^{k}L^iD_{\mathcal{X}^i}\norm{X_{t-1} - X_{t}} + \langle B^i_{s^i(k)} + V^i_{s^i(k)}, X^i_k - x^i_*\rangle\Big)\mathds{1}_{E^i_k} \\
& \overset{(b)}{\leq} -\langle F(X_k), X_k - x_*\rangle\mathds{1}_{E_k} - \sum_{i \in \mathcal{N}} \mathds{1}_{(E_k)^c\cap E^i_k} \langle \nabla_{x^i} J^i(X_k), X^i_k - x^i_*\rangle\\
& + \sum_{i \in \mathcal{N}}\Big(\sum_{t=s^i(k) + 1}^{k}L^iD_{\mathcal{X}^i}\norm{X_{t-1} - X_{t}} + \alpha_BD_{\mathcal{X}^i}\delta_{s^i(k)} \\
& \qquad + \langle V^i_{s^i(k)}, X^i_k - x^i_*\rangle\mathds{1}_{E^i_k} \Big),
\end{align*}
where the relation (a) follows from the $L^i$-Lipschitz continuity of $\nabla_x J^i$; the relation (b) can be derived by letting $D_{\mathcal{X}^i} \coloneqq \max_{x,y\in \mathcal{X}^i}\norm{x - y}$ and applying Lemma~\ref{le:bias}.
For each $\norm{X_{t-1} - X_t}$, it entails Lemma~\ref{le:md-lips} that
$\norm{X_{t-1} - X_{t}} = \norm{\nabla \psi^*(\nabla \psi(X_{t-1})) - \nabla \psi^*(\nabla \psi(X_{t-1}) - \gamma_{t-1}\hat{G}_{t-1})} \leq \gamma_{t-1}\norm{\hat{G}_{t-1}}_*/\Tilde{\mu} \leq \gamma_{t-1}\bar{g}_\star/\Tilde{\mu}$. 
Lemma~\ref{le:fdbk-prop} indicates that $s^i(k) \geq k - \bar{d}(k)$ for all $i$. 
Furthermore, since $X^i_k, \mathds{1}_{E^i_k} \in \Tilde{\mathcal{F}}_k$ while $V^i_{s^i(k)}$ is independent of $\Tilde{\mathcal{F}}_k$, $\expt{}{\langle V^i_{s^i(k)}, X^i_k - x^i_*\rangle\mathds{1}_{E^i_k} \mid \Tilde{\mathcal{F}}_k} = \langle \expt{}{V^i_{s^i(k)}\mid \Tilde{\mathcal{F}}_k}, X^i_k - x^i_*\rangle\mathds{1}_{E^i_k} = 0$.
If we further take the conditional expectation $\expt{}{\cdot \mid \Tilde{\mathcal{F}}_k}$ of both sides of the above inequality, it yields that
\begin{align*}
& \expt{}{-\langle \hat{G}_k, X_k - x_*\rangle \mid \Tilde{\mathcal{F}}_k} \leq -\langle F(X_k), X_k - x_*\rangle\mathds{1}_{E_k} \\
& + \sum_{i \in \mathcal{N}} \mathds{1}_{(E_k)^c\cap E^i_k} \bar{\nabla}_iD_{\mathcal{X}^i} + \beta_3\sum_{t=k-\bar{d}(k)+1}^{k}\gamma_{t-1} + \beta_4\delta_{k-\bar{d}(k)}, 
\end{align*}
where we let $\beta_3 \coloneqq \sum_{i \in \mathcal{N}}\frac{L^iD_{\mathcal{X}^i}\bar{g}_{\star}}{\Tilde{\mu}}$ and $\beta_4 \coloneqq \sum_{i \in \mathcal{N}} \alpha_BD_{\mathcal{X}^i}$. 
We then take $\expt{}{\cdot \mid \Tilde{\mathcal{F}}_k}$ of both sides of \eqref{eq:std-ineq} and apply the bound derived above to procure:
\begin{align*}
\begin{split}
& \expt{}{D(x_*, X_{k+1})\mid \Tilde{\mathcal{F}}_k} \leq D(x_*, X_{k})-\gamma_k\langle F(X_k), X_k - x_*\rangle\mathds{1}_{E_k} + \gamma_k \\
& \sum_{i \in \mathcal{N}}\mathds{1}_{(E_k)^c\cap E^i_k} \bar{\nabla}_iD_{\mathcal{X}^i} + \gamma_k\beta_3\sum_{t=k-\bar{d}(k)+1}^{k}\gamma_{t-1} + \beta_4\gamma_k\delta_{k-\bar{d}(k)} + \frac{\bar{g}_\star^2}{2\Tilde{\mu}}\gamma_k^2. 
\end{split}
\stepcounter{equation}\tag{\theequation}\label{eq:md-recr-ineq}
\end{align*}
Note that under Assumption~\ref{asp:delay-plus}$(\romannum{1})$, $(E^i_k)^c$ happens for only finitely many $k$, i.e., $\mathds{1}_{(E^i_k)^c} = 1$ for at most $\ceil{\bar{d} + 1}$ iterations, while under Assumption~\ref{asp:delay-plus}$(\romannum{2})$, $\mathds{1}_{(E_k)^c\cap E^i_k} \equiv 0$.
In either case, $\sum_{k \in \nset{}{+}}\gamma_k\sum_{i \in \mathcal{N}} \mathds{1}_{(E_k)^c\cap E^i_k} \bar{\nabla}_iD_{\mathcal{X}^i} < \infty$. 
For the next error term associated with delays, we have $\gamma_k\sum_{t=k-\bar{d}(k)+1}^{k}\gamma_{t-1} \leq \gamma_k\cdot \bar{d}(k) \gamma_{k-\bar{d}(k)} \propto O(k^{\alpha_d - 2\alpha_\gamma})$ and by choosing the parameters such that $2\alpha_\gamma - \alpha_d > 1$, we ensure that this term is summable. 
For the last two terms, trivially, $\gamma_k\delta_{k - \bar{d}(k)} \propto O(k^{-\alpha_{\gamma} - \alpha_{\delta}})$ and $\gamma_k^2 \propto O(k^{-2\alpha_{\gamma}})$, the summability of which follows from the assumptions $\alpha_{\gamma} + \alpha_{\delta} > 1$ and $\alpha_{\gamma} > 0.5$ imposed. 
By the Robbins-Siegmund theorem \cite[Thm.~1]{robbins1971convergence}, we arrive at the claims: 
$(\romannum{1})$ $D(x_*, X_{k})$ converges a.s. to a random variable that is finite a.s.;  
$(\romannum{2})$ $\sum_{k \in \nset{}{+}}\gamma_k\langle F(X_k), X_k - x_*\rangle\mathds{1}_{E_k} < \infty$ a.s. 
For each $\omega \in \Tilde{\Omega}$ with $\Tilde{\Omega}$ defined as a sample subset with probability one, by utilizing Lemma~\ref{le:insum-gamma}, i.e., $\sum_{k \in \nset{}{+}}\gamma_k\mathds{1}_{E_k}(\omega) = \infty$, we deduce that $\liminf_{k \to \infty} \langle F(X_k), X_k - x_*\rangle(\omega) = 0$. 
Thus, along a subsequence $(k_m)_{m \in \nset{}{+}}$, we have $\lim_{k \to \infty} \langle F(X_{k_m}), X_{k_m} - x_*\rangle(\omega) = 0$. 
By applying the boundedness of $\mathcal{X}$ and $X_k \in \mathcal{X}$, we can find a further subsequence $(\ell_m)_{m \in \nset{}{+}} \subseteq (k_m)_{m \in \nset{}{+}}$ such that $X_{\ell_m}(\omega) \to X_\star(\omega)$. 
Since $F$ is a continuous operator, $\lim_{m \to \infty} \langle F(X_{\ell_m}(\omega)), X_{\ell_m}(\omega) - x_*\rangle = \langle F(X_\star(\omega)), X_\star(\omega) - x_*\rangle = 0$. 
Since $x_*$ is a CP, $\langle F(x_*), X_\star(\omega) - x_*\rangle \geq 0$, which, together with the pseudo-monotone plus property of $F$, implies that $F(x_*) = F(X_\star(\omega))$. 
It then readily follows that for any $x \in \mathcal{X}$, $\langle F(X_\star(\omega)), x - X_\star(\omega)\rangle = \langle F(X_\star(\omega)), x - x_* + x_* - X_\star(\omega)\rangle \geq 0$, which implies that $X_\star(\omega)$ is also a CP. 
Then we can replace $x_*$ in \eqref{eq:md-recr-ineq} with $X_\star(\omega)$ and it follows that $D(X_\star(\omega), X_k)$ converges a.s.  
In addition, along the subsequence $(\ell_m)_{m \in \nset{}{+}}$, $D(X_\star(\omega), X_{\ell_m}(\omega)) \to 0$ by Assumption~\ref{asp:recip}. 
Therefore, $D(X_\star(\omega), X_k(\omega)) \to 0$ and we come to the conclusion that $X_k$ converges to a CP a.s. 
Thus, the convergence result also holds for the actual sequence of play $(\hat{X}_k)_{k \in \nset{}{+}}$ since $\delta_k \overset{k\to\infty}{\to} 0$.  
\end{proof}